\documentclass[12pt]{amsart}
	\usepackage{amsmath,amsfonts,amsthm} 
  \usepackage{amssymb,latexsym}  
  \usepackage[latin1]{inputenc}   
	\usepackage{hyperref}      
  \usepackage{fancyhdr}    
	\usepackage{url}
\hypersetup{bookmarksopen=true}

	\newtheorem{theorem}{Theorem}
	\newtheorem{lemma}{Lemma}
	
\title{  Quartic Forms in Many Variables} 
\author{ Jan H. Dumke }
\begin{document}
\begin{abstract} 
We show that a quartic $p$-adic form with at least $3192$ variables possesses a non-trivial zero. 
We also prove new results on systems of cubic, quadratic and linear forms. 
As an example, we show that for a system comprising two cubic forms $132$ variables are sufficient.
\end{abstract}
\maketitle	
\let\thefootnote\relax
\let\thefootnote\relax 

\section{Introduction}
Let $p$ be a rational prime and $F_1,\dots, F_r\in \mathbb{Q}_p[x_1,\dots,x_n]$ be  forms with respective degrees $d_1,\dots,d_r$. E. Artin conjectured in the
1930s, that $F_1, \dots , F_r$ have a common non-trivial zero provided
\begin{align} n> d_1^2+\dots+d_r^2.\notag \end{align} 
Unfortunately, this has been verified  merely for a single quadratic (Hasse \cite{Hasse}), a single cubic  (Lewis \cite{Lewis1}) and a system comprising two quadratic forms (Demyanov \cite{MR0080689} and independently Birch, Lewis and Murphy \cite{MR0136582}). In fact
counterexamples are known for many  $(d_1,\dots,d_r)$. 
Although false in general Bauer \cite{MR0013127} has shown there is a finite non-negative integer $v(d_1,\dots,d_r)$, independent of $p$, such that  $F_1,\dots, F_r$ possess a  non-trivial zero whenever \begin{align}n>v(d_1,\dots,d_r) \notag.\end{align} His proof reduces the problem to diagonal forms, which have been studied extensively (see in particular \cite{MR0153655}). Refined subsequent  results use quasi-diagonalisation techniques. 
The best general bound is due to Wooley \cite{MR1606240}. For a system comprising $r$ forms of degree $d$ he  showed that  $n>(rd^2)^{2^{d-1}}$  suffices. 
For a number of degrees better bounds are available. Firstly, we can extract better estimates from Wooley's proof for specifc $d$. Secondly, Heath-Brown \cite{MR2595750} considerably improved these for a single quartic  by establishing $v(4)\leq 4220$.   
His proof has been adapted by Zahid \cite{MR2541024} to show $v(5)\leq 4562911$ (Note that in this case the conjecture 
has been confirmed if $p>7$. See \cite{2013arXiv1308.0999D}). Heath-Brown's method provides better results if the involved degrees are not multiples of $p$. 
The purpose of this paper is to develop a variant yielding improved bounds if $p$ does divide the degree.  
\begin{theorem}\label{maintheorem1}
$v(4)\leq  3191$
\end{theorem} 
On the other hand, Terjanian \cite{MR552474} has constructed a dyadic quartic in $20$ variables which lacks  a non-trivial zero.
Several results address specific systems of forms. 
To put the next result into perspective it suffices to know that 
 $v(3,3)\leq 213$ can be derived by combing \cite{MR2595750} and \cite{MR2541024}.
\begin{theorem}\label{maintheorem2}
$ v(3,3)\leq  131$
\end{theorem} 
The proofs of  Theorem \ref{maintheorem1} and \ref{maintheorem2} rely on  results 
 for  systems comprising a  number of quadratic 
forms. 
 These allow us to impose certain constraints on the shape of the forms involved. By applying Hensel's Lemma we then establish a non-trivial zero.
Both theorems would enormously profit from better bounds on systems of quadratics. The method can be readily adapted for other degrees. 
\section{Preliminaries} 
We introduce a few Lemmas we shall need in due course. For ease of notation we write $V(r_3,r_2,r_1;p)$ for the least integer such that every system comprising $r_3$ cubic, $r_2$ quadratic and $r_1$ linear $p$-adic forms possesses a non-trivial zero as soon as $n>V(r_3,r_2,r_1;p)$. Concerning  quadratics the following bounds  as found in \cite{MR2595750} will be enough.
\begin{lemma}\label{sejviejfiea}
\begin{align}
V(0,r,0;p)\leq \notag
\begin{cases}
2r^2-16 & \text{if $r\geq 6$ is even}\\
2r^2-14 & \text{if $r\geq 7$ is odd}\\
\end{cases}
\end{align}
\end{lemma}
An estimate particular efficient for systems with just one cubic is due to Zahid \cite{MR2541024}.
\begin{lemma}\label{awnfkefnsjdsjdsej}Suppose $p\neq  3$ and $r_3\geq 1$. Then
\begin{align}
V(r_3,r_2,r_1;p)\leq  V(r_3-1,6(r_3-1)+r_2,9r_3+6r_2+r_1;p). \notag
\end{align}
\end{lemma}
We shall later need to establish that certain vectors are linearly independent. Heath-Brown \cite{MR2595750} provides an adequate criterion.  
\begin{lemma}\label{kbiuovsegz7inffdkjs}
Let $F\in \mathbb{Q}_p[x_1,\dots,x_n]$ be a form of degree $d$, having only the trivial zero in $\mathbb{Q}_p$. Let $\mathbf{e}_1,\dots,\mathbf{e}_k$ be linearly independent vectors in $\mathbb{Q}_p$, and suppose that we have a non-zero vector $\mathbf{e}\in \mathbb{Q}_p^n$ such that the form 
\begin{align}
F_0(t_1,\dots,t_k,t):=F(t_1\mathbf{e}_1+\dots+t_k\mathbf{e}_k+t\mathbf{e}),\notag
\end{align}
in the indeterminates $t_1,\dots,t_k$ and $t$, contains no terms of degree one in $t$. Then the set $\{\mathbf{e}_1,\dots,\mathbf{e}_k,\mathbf{e}\}$ is linearly independent.
\end{lemma}
We denote the $p$-adic valuation of $x\in \mathbb{Q}_p$ by $\nu(x)$. The following
non-standard variant of Hensel's Lemma is crucial to our proofs.
\begin{lemma}\label{akfjewajfiwjw}
Let $f$ be a polynomial over $\mathbb{Z}_p$ and suppose there exists an integer $x$ such that
 $\nu(f(x))\geq\nu(f'(x))^2$
 and, if equality holds, in addition $\nu(f''(x)\slash 2)\geq 1$. Then there exists a $p$-adic integer $y$ such that $f(y)=0$ and $y=x\pmod{p}$.
\end{lemma}
\section{Proof of Theorem \ref{maintheorem1}}
It is sufficient to establish the dyadic case, since Heath-Brown has shown that $313$ variables are enough if  $p$ is odd.
 As a first step we reduce this to a problem for a system of cubic, quadratic and linear forms. 
\begin{lemma} \label{awjdsfueesej}
$v(4)\leq  \max\{V(4,10,20;2),V(3,18,56;2)\}$
\end{lemma}
\begin{proof}[Proof of Lemma \ref{awjdsfueesej}] Let $F\in \mathbb{Q}_2[x_1,\dots,x_n]$ be a quartic form. Suppose that $F$ does  have the trivial zero only and \begin{align}\label{awjdsfueesej8w8328e2}n>\max\{V(4,10,20;2),V(3,18,56;2)\}.\end{align}
We shall construct a subspace of $\mathbb{Q}_2^n$ on which $F$ is of special shape. By applying Hensel's Lemma we then find a non-trivial zero.  In order to see how we can manipulate the shape assume that  $\mathbf{e}_1, \dots,\mathbf{e}_{k-1}\in \mathbb{Q}_2^n$ are linearly independent.  If $\mathbf{e}$ is an additional vector  we write
\begin{align}
F(&x_1\mathbf{e}_1+\dots+x_k\mathbf{e}_k+x\mathbf{e})\notag
= F(x_1\mathbf{e}_1+\dots+x_k\mathbf{e}_k)+\\&\sum_{\sum d_i=3}\mathbf{x}^{\mathbf{d}}L_{\mathbf{d}}(\mathbf{e})x+\sum_{\sum d_i=2}\mathbf{x}^{\mathbf{d}}Q_{\mathbf{d}}(\mathbf{e})x^2+\sum_{\sum d_i=1}\mathbf{x}^{\mathbf{d}}C_{\mathbf{d}}(\mathbf{e})x^3+F(\mathbf{e})x^4,\notag
\end{align}
 where $L_{\mathbf{d}}$ are linear, $Q_{\mathbf{d}}$ quadratic  and $C_{\mathbf{d}}$ cubic forms. 
If we want that some of these forms (and respective monomials) vanish, they  must have $\mathbf{e}$ as a common non-trivial zero.
This can be ensured at the cost of a  condition on $n$. 
If, in particular, $L_{\mathbf{d}}(\mathbf{e})=0$ for all $\mathbf{d}$ such that $\sum{d_i}=1$, then   $\mathbf{e}_1, \dots,\mathbf{e}_{k-1},\mathbf{e}$ are by Lemma \ref{kbiuovsegz7inffdkjs} linearly independent.\\ 
Thus we can   successively  choose vectors $\mathbf{e}_1, \dots,\mathbf{e}_5$ such that
\begin{align}
F(\mathbf{e}_1x_1+\dots+\mathbf{e}_5x_5)=F(\mathbf{e}_1)x_1^4+\dots +F(\mathbf{e}_5)x_5^4,\notag
\end{align}
 by imposing at most $4$ cubic, $10$ quadratic and  $20$ linear constraints (see (\ref{awjdsfueesej8w8328e2})). Clearly, $F(\mathbf{e}_i)\neq 0$ for all $i$.  We show that we may assume
\begin{align}\label{wejfliejfiwaiw}\nu(\mathbf{e}_1)=0,\qquad \nu(\mathbf{e}_2)=1,\qquad\nu(\mathbf{e}_3)=2.\end{align}
We say that a non-zero vector $\mathbf{e}\in \mathbb{Q}_2^n$ has level $r$ if $\nu(F(\mathbf{e}))= r \pmod{4}$.  
If $\mathbf{e}_1,\dots, \mathbf{e}_5$ have three different levels, then $(\ref{wejfliejfiwaiw})$ follows by relabelling and rescaling.
Otherwise we can find three vectors $\mathbf{e}_i$, $\mathbf{e}_j$,$\mathbf{e}_k$  of the same level $r$ for some $1\leq i<j<k\leq 5$. 
 As we may assume that $F(\mathbf{e}_i),F(\mathbf{e}_j),F(\mathbf{e}_k)\in \{ -2^r,2^r\}$, there are 
 $s,t\in  \{i,j,k\}$ such that $F(\mathbf{e}_s)+F(\mathbf{e}_t)= \pm 2^{r+1} \pmod{2^{r+2}}$.  
We replace $\mathbf{e}_s$ and $\mathbf{e}_t$ with $\mathbf{e}_s':=\mathbf{e}_s+\mathbf{e}_t$, which has level $r+1\pmod{4}$, and a newly chosen vector $\mathbf{e}_t'$. If the resulting  vectors still have two levels only, we 
repeat the argument until we obtain three vectors of different levels.\\ 
We  choose new vectors $\mathbf{e}_4$  of maximal level such that
 $F(\mathbf{e}_1x_1+\dots+\mathbf{e}_4x_4)$ is diagonal. We show that $\mathbf{e}_4$ has level $2$ at most. If $\mathbf{e}_4$ has level $3$, we  choose a new vector $\mathbf{e}_5$ such that $F(\mathbf{e}_1x_1+\dots+\mathbf{e}_5x_5)$ remains diagonal.
  By (\ref{wejfliejfiwaiw}) we may, after rescaling and relabelling, assume that 
\begin{align}\label{qwelri3qjiqiuq32}\nu(\mathbf{e}_1)=0,\quad \nu(\mathbf{e}_2)=1,\quad\nu(\mathbf{e}_3)=2,\quad\nu(\mathbf{e}_4)=3,\quad\nu(\mathbf{e}_5)=0.\end{align}
Thus we can set $x_1,x_5=1$ and pick $x_i\in \{0,1\}$ for $2\leq i\leq 4$ such that $F(\mathbf{e}_1x_1+\dots+\mathbf{e}_5x_5)=0\pmod{2^4}$.  The function $f(t):=F(\mathbf{e}_1t+\dots+\mathbf{e}_5x_5)$ then satisfies $\nu(f(1))\geq 4$, $\nu(f'(1))=2$ and $\nu(f''(1)\slash 2)\geq 1$. By Hensel's Lemma  $F$ has  a non-trivial zero, contrary to assumption. \\
We choose a new vector $\mathbf{e}_5$ of maximal level such that
\begin{align}
F(\mathbf{e}_1x_1+\dots+\mathbf{e}_5x_5)=&F(\mathbf{e}_1)x_1^4+F(\mathbf{e}_2)x_2^4+F(\mathbf{e}_3)x_3^4\notag \\&+F(\mathbf{e}_4)x_4^4+c_{45}x_4x_5^3+F(\mathbf{e}_5)x_5^4.\notag
\end{align}
By maximality $\mathbf{e}_5$ can be of level $2$ at most. We show that $\nu(F(\mathbf{e}_5))\leq 1$. Suppose, after rescaling, that $\nu(F(\mathbf{e}_4)),\nu(F(\mathbf{e}_5))=2$. By a case-by-case analysis of $\nu(c_{45})$ we establish a non-trivial zero.
If $\nu(c_{45})<\nu(F(\mathbf{e}_5))$ we use Hensel's Lemma to lift $\mathbf{e}_5$. In the case of $\nu(c_{45})=2$ it follows that  $2^{-2}F(\mathbf{e}_3+\mathbf{e}_4+\mathbf{e}_5)=0\pmod{2}$. Thus we can apply Hensel's Lemma to $f(t)=2^{-2}F(\mathbf{e}_3+\mathbf{e}_4+\mathbf{e}_5t)$.
If $\nu(c_{45})=3$ we may assume $\nu(\mathbf{e}_4+\mathbf{e}_5)>3$, since the case of four vectors of levels $0,1,2$ and $3$ has been discussed above (see (\ref{qwelri3qjiqiuq32})).  Thus we can choose $x_i\in \{0,2\}$ such that $ F(\mathbf{e}_1x_1+\mathbf{e}_2x_2+\mathbf{e}_3x_3+\mathbf{e}_4+\mathbf{e}_5)=0\pmod{2^7}$. Consequently, we can apply Hensel's Lemma to $f(t):=F(\mathbf{e}_1x_1+\mathbf{e}_2x_2+\mathbf{e}_3x_3+\mathbf{e}_4t+\mathbf{e}_5)$.  
 The case $\nu(c_{45})=4$ is slightly more involved. Since we may assume $\nu(\mathbf{e}_4+\mathbf{e}_5)>3$, it follows  that 
 $F(\mathbf{e}_3)=F(\mathbf{e}_i) \pmod{2^4}$ for some $i\in\{4,5\}$. 
 Thus  $\mathbf{e}_3':=\mathbf{e}_3+\mathbf{e}_i$ is a vector such that $\nu(\mathbf{e}_3')=3$. Assume without loss of generality that $i=4$. By (\ref{awjdsfueesej}) we can choose a new vector $\mathbf{e}_6$ such that
\begin{align}
F(\mathbf{e}_1x_1+&\mathbf{e}_2x_2+\mathbf{e}_3'x_3+\mathbf{e}_5x_5+\mathbf{e}_6x_6)=\notag \\
&F(\mathbf{e}_1x_1+\mathbf{e}_2x_2+\mathbf{e}_3'x_3+\mathbf{e}_5x_5)+F(\mathbf{e}_6)x_6^4.\notag
\end{align}
If $\nu(\mathbf{e}_6)=3$ then $\mathbf{e}_1,\mathbf{e}_2,\mathbf{e}_5,\mathbf{e}_6$ have levels $0,1,2$ and $3$, respectively, and we can proceed as in (\ref{qwelri3qjiqiuq32}). The same works for $\mathbf{e}_1,\mathbf{e}_2,\mathbf{e}_3',\mathbf{e}_6$ provided that $\nu(\mathbf{e}_6)=2$. If $\nu(\mathbf{e}_6)=1$ we  choose $x_1,x_3,x_5\in \{0,1\}$ such that $ F(2\mathbf{e}_1x_1+\mathbf{e}_2+\mathbf{e}_3'x_3+\mathbf{e}_5x_5+\mathbf{e}_6)=0 \pmod{2^5}$. If we set $f(t)=2^{-1}F(2\mathbf{e}_1x_1+\mathbf{e}_2+\mathbf{e}_3'x_3+\mathbf{e}_5x_5+\mathbf{e}_6t)$ then $\nu(f(1))\geq 4$, $\nu(f'(1))=2$ and $\nu(f''(1))\geq 1$ hold true and Hensel's Lemma can be applied. Similarly we can  find $x_2,x_3,x_5\in \{0,1\}$ such that $F(\mathbf{e}_1+\mathbf{e}_2x_2+\mathbf{e}_3'x_3+\mathbf{e}_5x_5+\mathbf{e}_6)=0 \pmod{2^4}$ provided $\nu(\mathbf{e}_6)=0$. Consequently, 
 Hensel's Lemma yields a non-trivial zero.\\
By (\ref{awjdsfueesej}) we can choose a new vector $\mathbf{e}_6$ of maximal level such that
\begin{align}
F(\mathbf{e}_1&x_1+\dots+\mathbf{e}_6x_6)=F(\mathbf{e}_1)x_1^4+F(\mathbf{e}_2)x_2^4+F(\mathbf{e}_3)x_3^4+F(\mathbf{e}_4)x_4^4\notag \\&+c_{45}x_4x_5^3+F(\mathbf{e}_5)x_5^4+c_{46}x_4x_6^3+c_{56}x_5x_6^3+c_{456}x_4x_5x_6^2+F(\mathbf{e}_6)x_6^4\notag.
\end{align} 
By maximality $\mathbf{e}_6$ has level $1$ at most. Suppose after rescaling that $\nu(\mathbf{e}_5),\nu(\mathbf{e}_6)=1$. If $\nu(c_{56})<1$, we can lift $\mathbf{e}_6$ via Hensel's Lemma. If $\nu(c_{56})=1$, there are $x_2,x_5,x_6\in \{0,1\}$ such that $2^{-1}F(\mathbf{e}_2x_2+\mathbf{e}_5x_5+\mathbf{e}_6x_6)=0\pmod{2}$ and one of it's partial derivatives does not vanish modulo $2$. Thus there exists a non-trivial zero and we may assume that $\nu(c_{56})\geq 2$. By maximality $\mathbf{e}_5+\mathbf{e}_6$ can not have level $2$ or $3$. Hence we can find $x_1,x_2\in \{0,2\}$ such that $F(\mathbf{e}_1x_1+\mathbf{e}_2x_2+\mathbf{e}_5+\mathbf{e}_6) =0 \pmod{2^9}$. Consequently,  we can apply Hensel's Lemma to $f(t):=F(\mathbf{e}_1x_1+\mathbf{e}_2x_2+\mathbf{e}_5t+\mathbf{e}_6)$.\\
By (\ref{awjdsfueesej8w8328e2}) we can choose  final vector  $\mathbf{e}_7$ of maximal level such that 
\begin{align}
F(\mathbf{e}_1x_1&+\dots+\mathbf{e}_7x_7)=F(\mathbf{e}_1)x_1^4+F(\mathbf{e}_2)x_2^4+F(\mathbf{e}_3)x_3^4+F(\mathbf{e}_4)x_4^4\notag \\
&+c_{45}x_4x_5^3+F(\mathbf{e}_5)x_5^4+c_{46}x_4x_6^3+c_{56}x_5x_6^3+c_{456}x_4x_5x_6^2\notag \\
&+F(\mathbf{e}_6)x_6^4+c_{47}x_4x_7^3+c_{57}x_5x_7^3+c_{67}x_6x_7^3+c_{457}x_4x_5x_7^2\notag \\&
+c_{467}x_4x_6x_7^2+c_{567}x_5x_6x_7^2+F(\mathbf{e}_7)x_7^4
\notag.
\end{align}
By  maximality  $\mathbf{e}_7$ has level $0$. Suppose that $\nu(\mathbf{e}_6),\nu(\mathbf{e}_7)=0$. If $\nu(c_{67})<0$, we lift $\mathbf{e}_7$. In case of $\nu(c_{67})=0$ we  can find $x_1,x_2\in \{0,1\}$ such that $F(\mathbf{e}_1x_1+\mathbf{e}_6x_6+\mathbf{e}_7x_7)=0\pmod{2}$  and Hensel's Lemma can be applied.
 Finally, suppose that $\nu(c_{67})>0$. Since $\mathbf{e}_6+\mathbf{e}_7$ can not have level $1,2,3$ we can find $x_1\in \{0,1\}$ such that $F(\mathbf{e}_1x_1+\mathbf{e}_6+\mathbf{e}_7)=0\pmod{2^8}$. Lifting this zero completes the proof of Lemma \ref{awjdsfueesej}.
\end{proof}
In order to estimate the quantities of $V(4,10,20;2)$  and $V(3,18,56;2)$ we provide an improved estimate.
\begin{lemma}\label{awnfkefnsej}Suppose $p=2\pmod 3$ and $r_3\geq 1$. Then
\begin{align}
V(r_3,r_2,r_1;p)\leq  V(r_3-1,3r_3+r_2,3r_3+3r_2+r_1;p). \notag
\end{align}
\end{lemma}
Theorem \ref{maintheorem1} now easily follows from Lemmas \ref{sejviejfiea}, \ref{awnfkefnsjdsjdsej}, \ref{awjdsfueesej} and \ref{awnfkefnsej}. 
\begin{proof}[Proof of Lemma \ref{awnfkefnsej}]
It is enough to show that $V(r_3,r_2,0;p)\leq V(r_3-1,3r_3+r_2,3r_3+3r_2;p)$. 
We choose a cubic form $C$ and  denote by $\mathbf{G}$ the system comprising all other forms. Suppose that for all non-zero $\mathbf{x}\in\mathbb{Q}_p$ such that $\mathbf{G}(\mathbf{x})=0$ we have $C(\mathbf{x})\neq 0$. Assume in addition that
\begin{align}\label{askfjkufdifisll}
n> V(r_3-1,3r_3+r_2,3r_3+3r_2;p).
\end{align}
By  (\ref{askfjkufdifisll}) there   exists $\mathbf{e}_1$ such that $C(\mathbf{e}_1x_1)=C(\mathbf{e}_1)x_1^4$ and $\mathbf{G}(\mathbf{e}_1x_1)$ is identically zero. We shall successively choose further vectors. A non-zero vector $\mathbf{e}$ is said to have level $r\in\mathbb{F}_3$ if $\nu(C(\mathbf{e}))=r\pmod{3}$.  Suppose we have chosen $s$ vectors $\mathbf{e}_1,\dots ,\mathbf{e}_s$ of different levels such that $C(\mathbf{e}_ix_i)=C(\mathbf{e}_i)x_i^4$ and $\mathbf{G}(\mathbf{e}_ix_i)=0$ for all $1\leq i \leq s$. Since $s\leq 3$ and (\ref{askfjkufdifisll}) we can 
choose an additional vector $\mathbf{e}_{s+1}$ such that
\begin{align}\label{skegjieiwqiure9qw}
C(\mathbf{e}_ix_i+\mathbf{e}_{s+1}x_{s+1})=C(\mathbf{e}_i)x_i^3+C(\mathbf{e}_{s+1})x_{s+1}^3
\end{align}
and $\mathbf{G}(\mathbf{e}_ix_i+\mathbf{e}_{s+1}x_{s+1})$ is identically zero 
for all $1\leq i \leq s$. It follows from Lemma \ref{kbiuovsegz7inffdkjs} that $\mathbf{e}_i$ and $\mathbf{e}_{s+1}$ are linearly independent for each $1\leq i \leq s$. By iterating this argument we find two vectors $\mathbf{e}_i, \mathbf{e}_j$ of the same level for some $1\leq i<j\leq 4$ such that $C(\mathbf{e}_ix_i+\mathbf{e}_{j}x_{j})$ is diagonal.  
After rescaling both the variables and the form we may assume that $\nu(C(\mathbf{e}_i)),\nu(C(\mathbf{e}_j))=0$. Since $p=2\pmod 3$, there exists $t \in \mathbb{F}_p$ such that $\nu(C(\mathbf{e}_it+\mathbf{e}_j))\geq 1$ and $\nu(C'(\mathbf{e}_it+\mathbf{e}_j))=0$. The Lemma then follows by applying Hensel's Lemma.
\end{proof}

\section{Proof of Theorem \ref{maintheorem2}}
We crucially establish a new bound if  $p=3$.
\begin{lemma} Suppose $r_3\geq 1$. Then
\begin{align}
V(r_3,r_2,r_1;3)\leq  V(r_3-1,3r_3+r_2,6r_3+3r_2+r_1;3). \notag
\end{align}
\end{lemma}
Theorem \ref{maintheorem2} now follows in conjunction with Lemmas \ref{sejviejfiea} and \ref{awnfkefnsjdsjdsej}. Also note the improvement provided by Lemma \ref{awnfkefnsej} if $p=2\pmod{3}$. 
\begin{proof}
It is suffices to prove that $V(r_3,r_2,0;3)\leq V(r_3-1,3r_3+r_2,6r_3+3r_2;3)$. 
We choose a cubic form $C$ and  denote by $\mathbf{G}$ the system comprising all other forms. Suppose that for all non-zero $\mathbf{x}\in\mathbb{Q}_3$ such that $\mathbf{G}(\mathbf{x})=0$ we have $C(\mathbf{x})\neq 0$. Assume in addition that 
\begin{align}\label{askfjkufdifis}
n> V(r_3-1,3r_3+r_2,6r_3+3r_2;3).
\end{align}
By (\ref{askfjkufdifis}) we can successively choose non-zero vectors $\mathbf{e}_1,\mathbf{e}_2,\mathbf{e}_3,\mathbf{e}_4$ such that 
\begin{align}\label{aefj932929re7ejd}
C(\mathbf{e}_1x_1+\dots +\mathbf{e}_4x_4)=C(\mathbf{e}_1)x_1^3+C(\mathbf{e}_2)x_2^3+C(\mathbf{e}_3)x_3^3+C(\mathbf{e}_4)x_4^3
\end{align} 
and $\mathbf{G}(\mathbf{e}_1x_1+\dots +\mathbf{e}_4x_4)$ is identical zero. By  Lemma \ref{kbiuovsegz7inffdkjs} are $\mathbf{e}_1,\mathbf{e}_2,\mathbf{e}_3$ and $\mathbf{e}_4$  linearly independent.  A vector $\mathbf{e}\in \mathbb{Q}_3-0$ is said to have level $r\in\mathbb{F}_3$ if $\nu(C(\mathbf{e}))=r\pmod{3}$.  Suppose there are two vectors $\mathbf{e}_i,\mathbf{e}_j$  of the same level for some $1\leq i<j\leq 4$. We rescale both  the variables and the form such that $\nu(\mathbf{e}_i),\nu(\mathbf{e}_j)=0$.
Since $C(\mathbf{e}_i),C(\mathbf{e}_j)=\pm 1\pmod{3}$ there exists $t_0\in \{1,-1\}$ such that
$3\mid C(\mathbf{e}_it_0+\mathbf{e}_j)$.  If we set $f(t)=C(\mathbf{e}_it+\mathbf{e}_j)$ either $\nu(f(t_0))\geq 2$, $\nu(f'(t_0))=1$ and $\nu(f''(t_0))\geq 1$ such that Hensel's Lemma can be applied or $\mathbf{e}_i':=\mathbf{e}_it_0+\mathbf{e}_j$ is of level $1$. Thus we can replace three vectors of the same level $r$ by two of level $r$ and $r+1$. We then choose an additional fourth vector.
By repeating this argument, relabelling and rescaling we find vectors $\mathbf{e}_i$, $\mathbf{e}_j$, $\mathbf{e}_k$ for some $1\leq i<j<k\leq 4$ such that $\nu(\mathbf{e}_i),\nu(\mathbf{e}_j)=0$ and $\nu(\mathbf{e}_k)=1$. We set  $ x_j=-C(\mathbf{e}_i)$ such that $3\mid C(\mathbf{e}_i+\mathbf{e}_jx_j)$. Thus we can write $C(\mathbf{e}_i+\mathbf{e}_jx_j)=3s$ and $C(\mathbf{e}_k)=3l$ where $3\nmid l$. We set $x_3=-sl$ if $s$ is a $p$-adic unit and $x_3=0$ otherwise. If we write $f(t)=C(\mathbf{e}_it+\mathbf{e}_jx_j+\mathbf{e}_kx_k)$, then $\nu(f(1))\geq 2$, $\nu(f'(1))=1$, $\nu(f''(1))\geq 1$ and Hensel's Lemma applies. 
\end{proof}
\subsection*{Acknowledgements} I wish to thank my supervisor Professor D.R. Heath-Brown for suggesting this topic, his encouragement and helpful comments.


\bibliographystyle{amsplain}
\bibliography{Bib}
 
\noindent\textsc{\small Mathematisches Institut, Bunsenstr. 3-5, 37073 Göttingen,}\\ \textsc{Germany}\\
\emph{ \small E-mail address:} \texttt{jdumke@uni-math.gwdg.de}

\end{document}